\documentclass[ejs]{imsart} 

\usepackage{amsmath}
\usepackage{amssymb}
\usepackage{amsthm}
\usepackage[utf8]{inputenc} 

\usepackage{graphicx} 
\usepackage{bbold}
\graphicspath{{figures/}}
\usepackage[usenames,dvipsnames,svgnames,table]{xcolor}
\usepackage{lmodern}

\usepackage{xspace}

\usepackage{paralist} 
\usepackage{mathrsfs}

\theoremstyle{remark}

\newcommand{\bbE}{\mathbb{E}}
\newcommand{\bbP}{\mathbb{P}}
\newcommand{\bbR}{\mathbb{R}}
\newcommand{\bbN}{\mathbb{N}}
\newcommand{\bbind}{\mathbb{1}}
\newcommand{\calG}{\mathcal{G}}
\newcommand{\calB}{\mathcal{B}}

\newcommand{\scrI}{\mathscr{I}}

\newcommand{\iid}{\textsc{iid}\xspace}
\newcommand{\pa}{\textsc{pa}\xspace}
\renewcommand{\qq}{\textsc{qq}\xspace}
\newcommand{\ee}{\textsc{ee}\xspace} 

\newcommand{\lrbkt}[1]{\{#1\}}

\input aad.tex
\definecolor{maroon}{rgb}{0.5,0,0}

\newcommand{\fn}[1]{{\color{purple}#1}\xspace}
\newcommand{\sss}{\scriptscriptstyle}

\renewcommand{\fn}[1]{{#1}\xspace} 

\usepackage{natbib}
\begin{document}
\begin{frontmatter}

\title{Consistent Estimation in General Sublinear Preferential Attachment Trees}
\runtitle{Consistent Estimation in the General PA Trees}

\begin{aug}
    \author{\fnms{Fengnan} \snm{Gao}\thanksref{t1}\ead[label=e1]{fngao@fudan.edu.cn}}
\thankstext{t1}{Research supported by the Netherlands Organization for Scientific Research and by the NNSF of China grant 1169010013.}
\affiliation{Fudan University}
\address{School of Data Science\\
    Fudan University\\
    Handan Road 220\\
Shanghai 200433, China}
\affiliation{Shanghai Center for Mathematical Sciences}
\address{Shanghai Center for Mathematical Sciences\\
    Handan Road 220\\
Shanghai 200433, China\\
\phantom{Email: \ }\printead*{e1}}
\author{\fnms{Aad} \snm{van der Vaart}\ead[label=e2]{avdvaart@math.leidenuniv.nl}\thanksref{t2}}
\thankstext{t2}{The research leading to these results has received funding from the European Research Council under ERC Grant Agreement 320637.}
\affiliation{Leiden University}
\address{Mathematical Institute \\
Leiden University \\
 Niels Bohrweg 1 \\
2333 CA Leiden, Netherlands\\
\phantom{Email: \ }\printead*{e2}}
\author{\fnms{Rui} \snm{Castro}\ead[label=e3]{rmcastro@tue.nl}}
\and
\author{\fnms{Remco} \snm{van der Hofstad}\ead[label=e4]{rhofstad@win.tue.nl}\thanksref{t3}}
\thankstext{t3}{Research supported by the Netherlands Organisation for Scientific Research (NWO) through VICI grant 639.033.806 and the Gravitation {\sc Networks} grant 024.002.003.}
\affiliation{Eindhoven University of Technology}
\address{
    Department of Mathematics\\
Eindhoven University of Technology\\
P.O. Box 513\\
5600 MB Eindhoven, Netherlands.\\
\phantom{Email:\ }\printead*{e3}
\phantom{Email:\ }\printead*{e4}
}
\end{aug}

\begin{abstract}
    We propose an empirical estimator of the preferential attachment function $f$ in the setting of general preferential attachment trees.  Using a supercritical continuous-time branching process framework, we prove the almost sure consistency of the proposed estimator.  We perform simulations to study the empirical properties of our estimators.
\end{abstract}

\end{frontmatter}
\section{Introduction}
\label{SectionIntroduction}
After the conception of the \textit{scale-free} phenomenon in Barab\'asi's series of seminal work
(\cite{Barabasi99emergenceScaling,barabasi1999mean,barabasi2000scale}), scientists from numerous
disciplines have made discoveries that support the ubiquity of real-world scale-free networks. 
Together with the notion of \textit{small-world} networks, these discoveries mark the emergence of
network science at the turn of the century. \textit{Preferential attachment} models, which received their
modern conception in \cite{Barabasi99emergenceScaling}, have become popular, because they are one of
the few {generative models} that give rise to scale-free behavior.

Consider the following \textit{dynamical} network model.
The network starts at stage $n=2$ with two nodes $v_1$ and $v_2$ connected by a single edge. Next it evolves recursively by adding
nodes $v_3,v_4\ldots$, which each connect by a single edge to a single node of the existing network. The incoming
nodes choose the node to which they connect by a probabilistic mechanism.
Given the network with nodes $v_1, \dots,v_n$ having \textit{degrees} $d_1(n),\ldots, d_n(n)$,
the node $v_{n+1}$ connects to the existing node $v_i\in\{v_1,\ldots, v_n\}$ with probability proportional
to $f\bigl(d_i(n)\bigr)$, i.e., with probability
$$\frac{f\bigl(d_i(n)\bigr)}{ \sum_{j=1}^n f(d_j(n))}.$$
Here $f:\bbN_+\to\bbR_+$ is a given function, which we call the \textit{preferential attachment} (\pa) \textit{function},
and we refer to $f(d_i(n))$ as the \textit{preference} for node $v_i$ at stage $n$. 
The \pa function $f$ is typically assumed to be non-decreasing, so that nodes of higher degrees inspire more incoming connections.
This explains the name preferential attachment model.  
After the incoming node $v_{n+1}$ has made its choice, the network evolves to the stage of $n+1$ nodes and the  scheme repeats
with the set of existing nodes $\{v_1,\dots,v_{n+1}\}$ and the incoming node $v_{n+2}$.  
The recursive procedure may be repeated to reach any number of nodes.  

The empirical degree distribution $P_k(n)$ is defined as the proportion of nodes of degree $k$ at time $n$:
	\[ 
	P_k(n) = \frac{1}{n} \sum_{i \in \{1,\ldots,n\}} \mathbb{1}_{\{d_i(n) = k\}}.
	\]
In the case that $f$ is affine with $f(k) = k + \delta$, it is well known that 
$P_k(n) \rightarrow p_k$ almost surely, as $n \rightarrow \infty$,  for any fixed $k$, for a limit 
$p_k$ that satisfies \( p_k \sim c_\delta k^{-3-\delta}\) as $k\rightarrow \infty$, for some constant 
$c_\delta$  (where $\sim$ means that the quotient of the two sides tends to 1; see \cite{mori2002random,hofstadcomplexnet}).
Thus in this scenario  the limiting degree distribution follows a power law with exponent $3+\delta$, 
and the \textit{scale-free} phenomenon occurs.  In particular, for $f(k) = k$
the exact asymptotic degree distribution can be worked out to be
$p_k = 4/(k (k+1) (k+2))$, and the power-law exponent is $3$ (\cite{mori2002random,bollobas2001degree}).

For \pa functions $f$ that are not affine, there are roughly two possible cases: super- and sublinear.  
If $f$ grows faster than linearly (the super-linear case; more precisely if $\sum_k 1/f(k)<\infty$),
then  one node will function as a hub (or a \textit{star}) that connects to a large fraction of all nodes, 
if not virtually to all the nodes (\cite{oliveira2005connectivity,krapivsky2001organization}). 
In this case the continuous-time branching process that describes the \pa tree 
(see Section~\ref{sec:rooted-tree}) no longer has a Malthusian parameter, and the precise behavior of the \pa model
is not well understood.  

In this paper we focus on the case of sublinear \pa functions $f$. This includes the affine case, but
also strictly sublinear cases,  which lead to a variety of possible limiting degree distributions. 
Next to power laws these include for instance power-laws with exponential truncation, which arise when the \pa
function $f(k)$ becomes constant for large enough $k$. In general, the limiting degree
distribution may be much lighter tailed than in the affine case, corresponding to rarer occurrence of nodes of high degrees
(see \cite{rudas2007random}).
Such scenarios have been reported frequently in empirical work on real-world networks.

The paper is concerned with the problem of {\em statistical estimation} of a  \pa function from an observed realization of a network.
Most empirical work to date has focused on estimating a power law exponent, which is presumed
to describe the data. However, despite the seeming omni-presence of power laws, studying them is often problematic.
The numbers of nodes of high degrees in an observed network usually exhibit large variations and irregular behavior as a function
of the degree. This is to be expected, as they are rare, but makes it hard to determine whether there is a power law at all.  
In applications where an affine \pa model and ensuing power law are in doubt, it might be wise to 
estimate the \pa function $f$ using a more general model instead. In this paper we focus on the set
of general sublinear \pa functions. Besides being of interest in itself, fitting a general sublinear function
will also shed light on the fit of an affine function, as the general estimate of $f$ may or may not resemble an affine function.  
In other words, estimating $f$ in our nonparametric setting will also help validate the modelling of power laws in \pa networks. 

The main contribution of our work is to propose an empirical estimator of a general \pa function,
and to prove its consistency.  Statistical estimation in a \pa network is not a 
conventional statistical problem that admits a standard asymptotic analysis, in two ways. First, although
Markovian, the growth of the tree in the \pa scheme (i.e., each transition of the Markov process)
depends on the full history of the evolution, where the dependence on the history may be long range.  \cite{7024931} even show
that the influence of the so-called \textit{seed graph}---the initial configuration from which the
preferential attachment graph starts to grow---does not vanish as the network size tends to
infinity.  Second, for practical purposes it is desirable to base a statistical estimator on the current day 
snapshot of the network and not on the evolution of the tree, as this will often not be oberved. Thus
we observe only the last realization of the Markov chain.
Perhaps the most surprising part in our study is that indeed one can consistently estimate a \pa function
from this final shapshot.  

The main mathematical tool of the paper is the theory of branching processes,
as introduced in this context by \cite{rudas2007random}.
Given nodes $v_1,v_2,\ldots,v_n$ with degrees $(d_i(n))_{i=1}^{n}$ and
preferences $\bigl(f(d_i(n))\bigr)_{i=1}^n$, the total preference $\sum_{i=1}^{n} f(d_i(n))$ is needed to
normalize the multinomial distribution on the nodes.  In the affine case that $f(k) = k + \delta$,
the total preference is deterministic and takes the form
\( \sum_{i=1}^{n} (d_i(n) + \delta) = n\delta + 2n. \)
This property allows to study the limiting degree distribution with simple recursions on the degree
evolution, and is also handy for the study of statistical estimators, as shown in \cite{gao2017spa}.  
However, in the case of general attachment functions, this ceases to hold and
the total preference is an involved random quantity that depends on the entire history of the
network evolution.  \cite{rudas2007random} overcome this difficulty by embedding
the \pa model in a  continuous-time branching-process framework, in which
each individual has children according to a pure birth process with birth rate $f(k)$,
if $k$ is the current number of chlidren. This embedding takes care of the normalization by the total preference. 
In the continuous-time dynamics the size of the network is random at any given time, 
but the process reduces to the \pa network at the stopping times where the tree reaches a given size.
\cite{rudas2007random} apply results  from the classical work of branching processes dating back to the 1970s and 1980s 
(see \cite{jagers1975branching,nerman1981convergence}) to prove that the empirical degree proportions $P_k(n)$ 
converge to the limits $p_k$ as in \eqref{eqn-pk-limit} almost surely, for any fixed $k$. 
We use similar arguments to derive the asymptotic consistency of our estimator.

This paper is organized as follows.  In Section~\ref{sec:main} we give the intuition behind the estimator and present our main result on its consistency.  Section~\ref{sec:rooted-tree} introduces the terminology of branching processes and gives a random tree model that is equivalent to the evolution of \pa networks.   We prove the main consistency result in Section~\ref{sec:network-ee-consistency}.  In the last section, we present
a simulation study on the performance of the proposed empirical estimators in different settings and discuss our observations.  Several simulation studies are carried out in order to uncover a more detailed picture of the properties of the empirical estimator, the most interesting one being that the estimator seems to be asymptotically normal with a $\sqrt{n}$ rate. 

\section[Empirical Estimator]{Construction of the Empirical Estimator and Main Result}
\sectionmark{Empirical Estimator}
\label{sec:main}
Suppose we have a \pa tree of $n$ nodes and $n$ is large enough so that the limiting degree distribution  $(p_k)_{k=1}^\infty$ is ``close'' to the empirical degree distribution $(P_k(n))_{k=1}^\infty$.  Suppose a new node comes in and needs to pick an existing node to attach to according to the \pa rule associated with the \pa function $f$.  Let $N_k(n)$ denote the number of nodes of degree $k$ (which is close to $n p_k$ in the limiting regime). Then, the probability of choosing an existing node of degree $k$ is
	\begin{equation*}
    	\frac{f(k) N_k(n)}{\sum_{j=1}^\infty f(j) N_j(n)} \approx \frac{f(k) p_k }{\sum_{j=1}^\infty f(j) p_j}.
	\end{equation*}
We are interested in the quantity $f(k)$ for each $k\ge 1$. However, $f$ is only identifiable up to scale and the denominator on the right hand side of the above display $ \sum_{j=1}^\infty f(j) p_j$ is a constant only depending on the \pa function $f$.
If we multiply the above display by an extra factor $n/N_k(n) \approx 1/p_k$, then we get a rescaled version of $f(k)$
	\[
    	\frac{f(k)}{ \sum_{j=1}^\infty f(j) P_j(n)} \approx \frac{ f(k)}{  \sum_{j=1}^\infty f(j) p_j},
	\]
where $P_j(n) = N_j(n)/n$ is the proportion of nodes of degree $j$ among all the $n$ nodes.
Henceforth, we define $r_k$ as the rescaled version of $f(k)$ for each $k$ and summarize the above heuristic as  
	\[
	r_k := \frac{f(k)}{\sum_{j=1}^\infty f(j) p_j} \approx \text{Probability of choosing a node of degree } k \times \frac{n}{N_k(n)}.
	\]
{Here we note that while $k\mapsto f(k)$ is not uniquely defined (as multiplying by a non-zero constant gives rise to the same attachment rules), the quantity $r_k$ \emph{is} unique as it is normalized such that $\sum_k r_kp_k=1$.}

We aim for an empirical estimator that mimics the above equation. This estimation, furthermore, should also work in the non-limiting regime.  Note that $n/N_k(n)$ is always readily available in any network, so it suffices to estimate the probability of the incoming node choosing an existing node of degree $k$. This probability can be estimated by counting the number of times that the incoming node chooses an existing node of degree $k$ during the evolution of the \pa network.  

{Let us now work the above heuristic out in a more formal way. Let $N_{\rightarrow k}(n)$ denote the number of times that a node is attached to a node of degree $k$. Further, denote the number of nodes of degree $k$ in the \pa network at time $n$ by $N_k(n)$.}
The empirical estimator (\ee) $\hat{r}_k (n)$ is then defined as
\begin{equation}
    \hat{r}_k(n) = \frac{N_{\rightarrow k}(n)}{N_k(n)}.
    \label{eqn-def-ee-rkt}
\end{equation}
Let $N_{>k}(n)=\sum_{j>k} N_j(n)$ denote the number of nodes of degree strictly larger than $k$ at time $n$.  For the \pa networks considered here, we have the following \textit{crucial} observation:

\begin{lemma} For all $k, n\geq 1$,
    $N_{\rightarrow k}(n) = N_{>k}(n). $
    \label{lemma-ee-crucial-observation}
\end{lemma}

\begin{proof}
    Observe that $N_{\rightarrow k}(n)$ counts the number of times that the incoming node chooses an existing node of degree $k$ to connect to up to time $n$.  Note that, if a node was chosen to be connected to as a node of degree $k$ at some point before time $n$, its degree at time $n$ is at least $k+1$.  
    On the other hand, we notice the node degree may only jump from $1$ to $2$, $2$ to $3$, \dots, $k$ to $k+1$, etc.  Therefore, if a node has degree strictly larger than $k$, it must have been chosen to be
    connected to as a node of degree $k$ at some time.  
    This gives the equality as in the statement of the lemma.
\end{proof}

In the light of the above observation, we note that \eqref{eqn-def-ee-rkt} is equivalent to
\begin{equation}
    \hat{r}_k(n) = \frac{N_{>k}(n)}{N_k(n)}.
    \label{eqn-def-eqv-ee}
\end{equation}
{
We give the main result of this paper in the following theorem, which applies to \pa functions satisfying the following
condition. Given a function $f: \bbN_+\to[0,\infty)$ define a function $\rho_f: (0,\infty)\to [0,\infty]$  by 
\begin{equation}
\label{EqDefRhof}
\rho_f(\lambda)=\sum_{l=1}^\infty\prod_{k=1}^l \frac{f(k)}{\lambda+f(k)}.
\end{equation}
Then the theorem assumes that the range of $\rho_f$ contains an open neighbourhbood of 1. 
In Section \ref{sec:rooted-tree} below we note that this condition is satisfied
by most \textit{sublinear} \pa functions $f$, whereas for \pa functions that increase faster than linear $\rho_f(\lambda)$ will typically
be infinite for every $\lambda>0$ and the condition fails.}

\begin{theorem}
 \label{thm-consist-main}
{
If the range of the function $\rho_f$ attached to the true \pa function $f$ contains an open neighhourhood of 1,}
then the estimator $\hat{r}_k(n)$ defined in \eqref{eqn-def-eqv-ee} is consistent almost surely, i.e., for any $k$,
    \begin{equation}
        \hat{r}_k(n) \xrightarrow{\text{a.s.}} r_k, \quad \text{as } n\rightarrow \infty,
        \label{eqn-as-consistency}
    \end{equation}
where $\xrightarrow{\text{a.s.}}$ denotes convergence almost surely.
\end{theorem}

The proof of the theorem is deferred to Section~\ref{sec:network-ee-consistency}. 

The estimator  \eqref{eqn-def-eqv-ee} can be computed from the network as observed at time $n$, without needing access
to the evolution of the network up to this time. This is important when modelling a real-world network as a \pa network, 
because in real-world applications it is often difficult, 
expensive, or even impossible to recover the evolution history of a network. 


The form of the empirical estimator $ \hat{r}_k(n)$ in \eqref{eqn-def-eqv-ee} sparks some
philosophical considerations.  Consider the degree of a node as a measure of its wealth, i.e., the
more neighbors, the richer.  Suppose that a node of degree $k$ asks how likely it is
to get \textit{richer}, i.e., to receive an extra connection.  This is equivalent to asking for an
estimate of $f(k)$.  The estimator \eqref{eqn-def-eqv-ee} counts the
number of the richer nodes $N_{>k}(n)$ and the number of nodes at the same level $N_k(n)$, and 
returns the quotient of these numbers as an 
estimate of  $f(k)$ (up to normalization).  If you live in the world of these nodes and wonder about  your
chance of moving up, then you might naturally come up with the aforementioned ratio.  The higher the
number of people above you relative to the number of people sharing your rank, the better 
your chance to move up.

\section[Branching Process]{Borrowing strength from branching processes}
\sectionmark{Branching Process}
\label{sec:rooted-tree}
In this section we introduce the terminology needed to reformulate the 
\pa model in the language of branching processes, similar to \cite{rudas2007random}.  As the \pa function is no longer affine, the conventional (and somewhat elementary) techniques, e.g., the martingale method as in \cite[Chapter 8]{hofstadcomplexnet}, do not work anymore.  
However, the supercritical branching processes 
observed at the random stopping times where their size is fixed,
designed first by \cite{rudas2007random}, are (almost) equivalent to the \pa trees, which in turn enables us to study the \pa trees with the
well-established results of general branching processes.

\subsection{Rooted ordered tree}
The \pa network is a \textit{rooted ordered tree}, which can be described as an evolving genealogical
tree, where the nodes are individuals and the edges are parent-child relations.  The usual notation
for the nodes are $\emptyset$ for the root of the tree and $l$-tuples $(i_1,\dots,i_l)$ of positive natural numbers $i_j\in\bbN_{+}$
for the other nodes ($l\in\bbN_{+}$). The children of the root are labeled $(1), (2),\ldots$, and in general 
$x = (i_1,\dots,i_l)$ denotes the $i_l$-th child of the $i_{l-1}$-th child of $\cdots$ of the $i_1$-th child
of the root.  Thus the set of all possible individuals is 
\begin{equation*}
    \mathscr{I}= \lrbkt{\emptyset} \cup \left( \bigcup_{l=1}^\infty \mathscr{I}_l\right),
\qquad \mathscr{I}_l = \bbN_{+}^l.
\end{equation*}
For $x = (x_1, \dots, x_k)$ and $y = (y_1, \dots, y_l)$ the notation $xy$ is shorthand for 
the concatenation $(x_1, \dots, x_k, y_1, \dots, y_l)$, and, in particular, $xl=(x_1, \dots, x_k, l)$.

Since the edges of the tree can be inferred from the labels of the nodes $(i_1,\ldots, i_l)$,
a rooted ordered tree can be identified with a subset $G \subset \scrI$.
(Not every subset corresponds to a rooted ordered tree, as the labels need to 
satisfy the  \textit{compatibility conditions} that  for every $(x_1, \dots, x_k) \in G$ we have
$(x_1,\dots, x_{k-1})\in G$ (\textit{parent must be in tree}) as well as $(x_1, \dots, x_k-1) \in G$ if
$x_k \ge 2$ (\textit{older sibling must be in tree}).)  The set of all finite rooted ordered trees is denoted by
$\calG$.  In this terminology and notation the \textit{degree} of node $x \in G$ that is not the root is
the number of its children in $G$ plus $1$ (for its parent), given by 
\begin{equation}
    \deg (x,G) =|\lrbkt{l \in \bbN_{+}\mid xl\in G }|+1.
    \label{eqn-ee-deg-def}
\end{equation}

\subsection{Branching process}
\label{subsec:ee-branching-process}
The evolution in time of the genealogical tree is described through stochastic processes $\bigl( \xi_x(t) \bigr)_{t\ge 0}$, one for each 
individual $x\in \mathscr{I}$. The random point process $\xi_x$ on $[0,\infty)$, given the ages of the parent $x$ at the births of its children, describes the node $x$ giving births to its children. 
The birth time $\sigma_x$ of individual $x$ in calendar time is
defined recursively, by setting $\sigma_ \emptyset=0$ (the root is born at time zero) and
\[
    \sigma_y = \sigma_x + \inf\{u\ge 0\colon \xi_x(u)\ge l \},\qquad \text{ if }y = xl.
\]
Thus the $l$-th child of $x$ is born at the birth time of $x$ plus the time of the $l$-th event in $\xi_x$.

It is assumed that the birth processes $\xi_x$ for different $x\in \mathscr{I}$ are \iid. 
{This is the defining property of a continuous-time branching process.}
Formally, we may define all processes $\xi_x$ on the product probability space 
\begin{equation*}
    (\Omega, \calB, P) =  \prod_{x \in \scrI} (\Omega_x, \calB_x, P_x),
\end{equation*}
where every $(\Omega_x, \calB_x, P_x) $ is an independent copy of 
a single probability space $(\Omega_0, \calB_0, P_0)$ and every $\xi_x$ is defined as
$\xi_x(\omega)=\xi(\omega_x)$ if $\omega=(\omega_x)_{x\in\scrI}\in\Omega$, for
$\xi$ a given point process defined on $(\Omega_0, \calB_0, P_0)$. 
We identify the point process $\xi$ with the process $\xi(t)$ giving the number of points in $[0,t]$, for $t\ge 0$,
and write $\mu(t) = \bbE[\xi(t)]$ for its intensity measure, which is often called the {\em reproduction function} in this
context.

Besides the reproductive process $\xi_x$ we also attach a \textit{random characteristic} $\phi_x$ to every
individual $x\in \mathscr{I}$. This is also a stochastic process $\bigl(\phi_x(t)\bigr)_{t\ge 0}$, which we take non-negative,
measurable and separable. For simplicity, we define $\phi_x(t) =0 $ for  $t <0$.  We then proceed to define
\begin{equation*}
    Z_t^{\phi} = \sum_{x \in \scrI\colon \sigma_x\le t} \phi_x(t- \sigma_x).
\end{equation*}
If $\phi_x(t)$ is viewed as a characteristic of individual $x$ when $x$ has age $t$,
then the variable $\phi_x(t-\sigma_x)$ is the characteristic of individual $x$ at calendar time $t$, and
$Z_t^\phi$ is the sum of all such characteristics over the individuals that are alive at time $t$ 
(i.e., {individuals $x$ for which} $\sigma_x\le t$).

The characteristics $\phi_x$ are assumed independent and identically distributed
for different individuals $x$, as the reproductive processes. Formally this may be achieved by defining $\phi_x(\omega)=\phi(\omega_x)$
if $\omega=(\omega_x)_{x\in\scrI}\in\Omega$, for a given stochastic process $\phi$ on $(\Omega_0, \calB_0, P_0)$. 
This allows the two processes $\xi_x$ and $\phi_x$ attached to a given individual to be dependent.
In fact, we shall be interested in the choices, for a given natural number $k$,
\begin{equation}
\label{EqCharacteristics}
\begin{aligned}
\phi(t)&\equiv 1,\\
\phi(t)&=\bbind_{\{\xi(t)= k - 1\}},\\
\phi(t)&=\bbind_{\{\xi(t)> k- 1 \}}.
\end{aligned}
\end{equation}
For the first characteristic the variable $Z_t^{\phi}=Z_t^1$ is equal to 
the total number of individuals born up to time $t$, for the second it equals the total number of those individuals
with exactly $k - 1$ (and hence of degree $k$), and for the second more than $k - 1$ children at time $t$ (hence of degree $>k$). 
As we will soon see, we are interested in the degree distribution of the network and the degree of a node is defined to be its number of children ($\xi(t)$) plus one, this is why $k - 1$ appears instead of $k$.  

We consider \textit{supercritical}, \textit{Malthusian} branching processes for which the following three conditions hold:
\begin{enumerate}
    \item $\mu$ does not concentrate on any lattice $\lrbkt{0, h , 2h,\dots}$ for $h >0$.  
    \item There exists a number $\lambda^*>0 $ such that
        \begin{equation}
            \int_0^{\infty} {\mathrm e}^{-\lambda^* t} \mu(dt) = 1.
  \label{eqn:network-ee-malthusian}
        \end{equation}
    \item {
        The second moment of ${\mathrm e}^{-\lambda^* t}\, \mu(dt)$ is finite, i.e.,
        \begin{equation}
            \int_0^\infty t^2 \mathrm{e}^{-\lambda^* t}\, \mu(dt) < \infty.
            \label{eqn:network-ee-lambda-finite}
        \end{equation}}
\end{enumerate}
Condition \eqref{eqn:network-ee-malthusian} is the \textit{Malthusian} assumption, and $\lambda^*$
is called the Malthusian parameter.

The following combines Theorems~5.4 and~6.3 of \cite{nerman1981convergence}
(also compare Theorem A from \cite{rudas2007random}).

{
\begin{proposition}
\label{prop-fraction-as}
Consider a supercritical, Malthusian branching process with Malthusian parameter $\lambda^*$ as in \eqref{eqn:network-ee-malthusian} and 
satisfying \eqref{eqn:network-ee-lambda-finite}, 
and two associated bounded characteristics $\phi$ and $\psi$.  
Then there exists a random variable $Y_\infty$ depending on $\xi$ only
such that almost surely on the event that the total population size $Z_t^1$ tends to infinity, as $t\ra\infty$,
\begin{equation}
\mathrm{e}^{-\lambda^*t} Z_t^\phi\ra Y_\infty    \frac{\int_0^\infty {\mathrm e}^{-\lambda^*t}\bbE [\phi(t)]\,dt}
{\int_0^\infty t {\mathrm e}^{-\lambda^*t}\,\mu(dt)},    
 \label{eqn-exponential convergence}
  \end{equation}
Furthermore, if $\bbE [{}_{\lambda^*}\xi(\infty)\log_+{}_{\lambda^*}\xi(\infty)]<\infty$, for
${}_{\lambda}\xi(t)=\int_0^t\mathrm{e}^{-\lambda s}\,\xi(ds)$, then $Y_\infty$ can be taken strictly positive, and 
almost surely on the event that the total population size $Z_t^1$ tends to infinity,
\begin{equation}
        \frac{ Z_t^\phi}{ Z_t^\psi} \xrightarrow{\text{a.s.}} \frac{\int_0^\infty \mathrm{e}^{-\lambda^*t}\bbE [\phi(t)]\,dt}
{\int_0^\infty \mathrm{e}^{-\lambda^*t}\bbE [\psi(t)]\,dt},
        \label{eqn-fraction-as}
\end{equation}
as $t\rightarrow\infty$. The convergence \eqref{eqn-fraction-as} is true, more generally, if
$\int_0^\infty \mathrm{e}^{-\lambda t} \mu( dt) < \infty$, for some number $\lambda<\lambda^*$.
\end{proposition}

\begin{proof}
Assertion \eqref{eqn-exponential convergence} follows by Theorem 5.4 of \cite{nerman1981convergence}.
Note that Condition 5.1 of the latter theorem is satisfied because of
\eqref{eqn:network-ee-lambda-finite} (see (5.7) of \cite{nerman1981convergence}), while Condition 5.2
follows by the assumed boundedness of $\phi$ and $\psi$. 
By Proposition 1.1 and (3.10) in the same reference, the convergence of the integral 
$\bbE {}_{\lambda^*}[\xi(\infty)\log_+{}_{\lambda^*}\xi(\infty)]$ implies that 
the variable $Y_\infty$ is strictly positive almost surely on the event that the total population size $Z_t^1$ tends to infinity. Then 
\eqref{eqn-fraction-as} follows by applying \eqref{eqn-exponential convergence} to both $e^{-\lambda^*t}Z_t^\phi$ and $e^{-\lambda^*t}Z_t^\psi$, and 
taking the quotient. The last assertion of the proposition follows from Theorem 6.3 of \cite{nerman1981convergence}.
\end{proof}
}

%

\subsection{The continuous random tree model}
\label{subsec:continuous-tree}
To connect back to the \pa model, given a \pa function $f$, we now define the process $\xi$ as
a pure birth process with birth rate equal to $f(\xi(t)+1)$, i.e., the continuous-time Markov process with state space
being the non-negative integers and the only possible transitions given by
\begin{equation}
    P(\xi(t+dt) = k+1 \mid \xi(t) = k) = f(k+1)\, dt + o(dt).
    \label{eqn-birth-process}
\end{equation}
The genealogical tree is then also a Markov process on the state space $\calG$. The initial state
of the process is the root $\lrbkt{\emptyset}$ of the tree, and the jumps of this process 
 correspond to an individual $x\in G$ giving
birth to a child, which is then incorporated in the tree as an additional node. In the preceding notation this means that the
process can jump from a state $G$ to a state of the form $G \cup \lrbkt{xk}$, 
where necessarily $x \in G$ and $k = \deg(x,G)$ is the
number of children that $x$ already has in the tree plus 1.  This  jump 
is made with rate $f(\deg(x,G))$, since according to \eqref{eqn-birth-process} with $\xi=\xi_x$
the individual $x$ gives birth to a new child with rate $f(k)$ if $x$ already has $k-1$
children. The description in terms of rates means more concretely
that given the current state $G$, the Markov process can jump to the finitely many 
possible states $G\cup\lrbkt{xk}$, $x \in G$ and $k = \deg(x,G)$, and it chooses between these states with probabilities 
\begin{equation*}
    \frac{f(\deg(x,G))}{\sum_{y \in G} f(\deg(y,G))},\qquad x\in G.
\end{equation*}
Furthermore, the waiting time in state $G$ to the next jump is an exponential variable with intensity equal to
the \textit{total preference} 
\begin{equation*}
    F(G) = \sum_{x\in G} f(\deg(x,G)).
\end{equation*}
The continuous-time scale of the process is not essential to us, but it is convenient for our calculations. 
We shall use that when $t\rightarrow\infty$ the continuous-time tree visits the same states (trees) as the \pa model,
and taking limits as $t\ra\infty$ is equivalent to taking limits in the \pa model as the number of nodes
increases to infinity almost surely.

In order to apply Proposition~\ref{prop-fraction-as} in our setting we need to verify its conditions
on the birth process $\xi$ and the reproduction function $\mu(t)=\bbE[\xi(t)]$, and determine the Malthusian parameter.
The events of the pure birth process \eqref{eqn-birth-process} can be written as $T_1<T_1+T_2<T_1+T_2+T_3<\cdots$,
where $(T_k)_{k=1}^\infty$ are independent random variables exponentially distributed with rates $(f(k))_{k=1}^\infty$.
The total number of births $\xi(t)=\int \bbind_{(0,t]}(u)\,\xi(du)$ at time $t$ is equal to $\sum_{l=1}^\infty \bbind_{(0,t]}(T_1+\cdots+T_l)$, 
which clearly tends to infinity almost surely as $t\ra\infty$. Furthermore, we have 
$\int \mathrm{e}^{-\lambda u}\,\xi(du)=\sum_{l=1}^\infty \mathrm{e}^{-u(T_1+\cdots+T_l)}$, and hence
	\begin{equation}
	\label{laplace-transform}
	 \int_0^\infty \mathrm{e}^{-\lambda u}\,\mu(du)=\bbE \Bigl[\sum_{l=1}^\infty \mathrm{e}^{-\lambda(T_1+\cdots+T_l)}\Bigr]
	=\sum_{l=1}^\infty \prod_{k=1}^l \frac{f(k)}{\lambda+f(k)}=\rho_f(\lambda),
	\end{equation}
{
for $\rho_f$ defined in \eqref{EqDefRhof}. The Malthusian parameter $\lambda^*$ 
is defined as the argument where the function $\rho_f$ in the display equals one. The terms of the series definining
this function are nonnegative, and strictly decreasing
and convex in $\lambda$. Hence if $\rho_f$ is finite for some $\lambda>0$, then it is finite and continuous on an
interval $(\underline\lambda, \infty)$ and tends to zero as $\lambda\ra\infty$, by the dominated convergence theorem.
If $\lim_{\lambda\downarrow\underline\lambda} \rho_f(\lambda)>1$, then $\lambda^*$ exists and is interior
to the interval $(\underline\lambda, \infty)$. It is shown in Lemma~1 on page~200 of \cite{rudas2007random} that 
in this case the associated birth process $\xi$ satisfies  $\bbE [{}_{\lambda^*}\xi(\infty)]^2<\infty$ as soon as
$f(k)\ra\infty$. 

Thus all conditions of Proposition~\ref{prop-fraction-as} are satisfied if
$\lim_{\lambda\downarrow\underline\lambda} \rho_f(\lambda)>1$, and this is equivalent to 1 being an inner
point of the range of $\rho_f$.

\begin{example}
If $f$ is strongly sublinear in the sense that $f(k)\le (k+\delta)^{\beta}$, for some $0<\b<1$ and every $k\in\bbN_+$, then 
$\rho_f(\lambda)<\infty$ for every $\lambda>0$, and hence $\underline\lambda=0$. By the monotone convergence 
theorem $\rho_f(\lambda)\uparrow\rho_f(0)=\infty$, as $\lambda\downarrow 0$. Hence
all conditions of Proposition~\ref{prop-fraction-as} are satisfied.

To see that  $\rho_f$ is finite on $(0,\infty)$, note that $\rho_f$ is increasing in $f$, so that
it suffices to prove the claim in the special case that $f(k)= (k+\delta)^{\beta}$.
Since $\log(1+x)>x/2$, for $x\in [0,1]$, we have for $K$ such that  $\min_{k\ge K}f(k)\ge \lambda$,
$$\prod_{k=K}^l \frac{f(k)}{\lambda+f(k)}=\exp\Bigl[-\sum_{k=K}^l\log \Bigl(1+\frac{\lambda}{f(k)}\Bigr)\Bigr]
\le \mathrm{e} ^{-(1/2)\sum_{k=K}^l\l/f(k)}.$$
For $f(k)=(k+\delta)^\beta$, we can choose $K=\lambda^{1/\beta}$ and then find that the right side
is bounded above by a multiple of $\mathrm{e}^{-\lambda(l+\d)^{1-\beta}/(2-2\beta)}$. This sums finitely with
respect to $l$, and hence $\rho_f(\lambda)<\infty$.
\end{example}

\begin{example}
In the exact linear case $f(k)=k+\delta$ for some $\d\ge 0$, the function $\rho_f$ can be computed exactly as 
\fn{$\rho_f(\lambda)=(1+ \delta)/(\lambda-1)$}
(cf., Section~4.2 of \cite{rudas2007random}).
Then $\underline\lambda=1$ and $\rho_f(\lambda)$ increases to infinity as $\lambda\downarrow\underline\lambda$,
and all conditions of Proposition~\ref{prop-fraction-as} are satisfied.

If $f$ is sublinear in the sense that $f(k)\le k+\delta$, for some $\delta\ge 0$ and every $k\in\bbN_+$, then 
$\rho_f(\lambda)<\infty$ for every $\lambda>1$, but the behaviour of the function
$\rho_f$ near $\underline\lambda$, which will be contained in $[0,1]$, appears to depend on the fine properties of $f$. 
\end{example}

When applied to the root node $x=\emptyset$, the right side of formula \eqref{eqn-ee-deg-def} gives the degree plus 1
and not the degree, as the root $v_1$ in the \pa model in Section~\ref{SectionIntroduction}
does not have a parent. The preceding branching model then can be viewed as the description of an alternative
\pa model in which the root is also considered to have a parent (say $v_0$), 
whom, however, will never give birth to other children than the initial one. 
This approach is followed by \cite{rudas2007random}. Alternatively, to match the continuous-time branching process exactly to
the original description of the \pa model, the birth process $\xi_\emptyset$ of the root should be defined by 
\eqref{eqn-birth-process} but with $f(k)dt+o(dt)$ instead of $f(k+1)dt+o(dt)$ on the right side. Intuitively,
this replacement makes little difference for our asymptotics. However, as the birth processes will then not be identically distributed,
a direct reference to \cite{nerman1981convergence} will be impossible. Below we solve this by running two separate,
independent branching processes from each of the two starting nodes $v_1$ and $v_2$.}

\section[Consistency]{Consistency of the Empirical Estimators}
\sectionmark{Consistency of the EE}
\label{sec:network-ee-consistency}
For completeness, we present a result without proof from  \cite{rudas2007random} giving the limiting degree distribution
of the \pa model.

\begin{proposition}
    \label{prop-ee-limit-deg-dist}
{
If the range of the function $\rho_f$ attached to the true \pa function $f$ contains an open neighhourhood of 1,}
then as $n\rightarrow \infty$, the empirical degree distribution  $P_k(n)$ converges almost surely for any $k$ to some limit $p_k$, i.e.,
    \begin{equation}
        P_k(n) \xrightarrow{\text{a.s.}} p_k = \frac{\lambda^*}{\lambda^* + f(k)}\prod_{j=1}^{k-1} \frac{f(j)}{\lambda^* + f(j)}, \quad \forall k\in \bbN_+,
        \label{eqn-pk-limit}
    \end{equation}
where the empty product is defined to be $1$, so that $p_1 = \lambda^*/\bigl(\lambda^* + f(1)\bigr)$. 
\end{proposition}


It follow from  \eqref{eqn-pk-limit}  that $p_{k+1} = p_k  f(k) / (\lambda^{*} + f({k+1}))$.  Furthermore, 
\begin{equation}
    \label{eqn-lambda-star-sum}
\begin{aligned}
    \sum_{k=1}^\infty f(k) p_k & = \sum_{k=1}^\infty \frac{\lambda^* f(k)}{ \lambda^* + f(k)} \prod_{i=1}^{k-1} \frac{f(i) }{ \lambda^* + f(i)} \\
    & = \sum_{k=1}^\infty \lambda^* \prod_{i=1}^{k} \frac{f(i)}{ \lambda^* + f(i)} =\lambda^*\rho_f(\lambda^*)= \lambda^*.
\end{aligned}
\end{equation}

\begin{proof}[Proof of Theorem~\ref{thm-consist-main}]
{
We embed the evolution of the \pa model within the continuous time branching process framework described
in Section~\ref{sec:rooted-tree}. As explained at the end of the latter section a straightforward embedding
gives a slightly different \pa model, in which the degree of the root node is counted one higher than in the
original \pa model. We first give the proof for the adapted \pa model, and next
explain how this proof can be adapted to treat the original \pa model.

The degree $\deg(x,t)$ of node $x$ in the continuous time branching tree at time $t$ relates
to its associated reproductive process as $\deg(x,t) = \xi_x(t-\sigma_x) + 1$. 
Therefore the total number of nodes of degree strictly larger than $k$ present in the tree at time $t$, and the total number
of nodes of degree $k$ are given by 
$$\sum_{x\in \scrI} \bbind_{\{ \xi_x(t - \sigma_x) + 1 > k\} }, \qquad\text { and } \qquad
\sum_{x\in \scrI}  \bbind_{\{ \xi_x(t - \sigma_x) + 1 = k\} }.$$
These are the processes $Z_t^\phi$ and $Z_t^\psi$ corresponding to the characteristics
$\phi(t) =  \bbind_{\lrbkt{\xi(t) + 1 > k}}$ and $\psi(t) = \bbind_{\lrbkt{ \xi(t) + 1 = k}}$, respectively. 
It follows that 
\begin{equation}
\frac{|\lrbkt{ x\colon  \deg(x,t)>k}|}{ |\lrbkt{ x\colon  \deg(x,t)=k}|} 
=\frac{Z_t^\phi}{Z_t^\psi}
\ra\frac{\int_{0}^\infty {\mathrm e}^{-\lambda^* t} \bbP(\xi(t) + 1 >k)\, dt}
{\int_{0}^\infty {\mathrm e}^{-\lambda^* t} \bbP(\xi(t) + 1 =k)\, dt},
 \label{eqn-limit-phi-psi}
\end{equation}
almost surely, as $t\ra\infty$, by the second assertion of Proposition~\ref{prop-fraction-as}. When evaluated at the (random) time $t$ such that
the total population size $Z_t^1$ is equal to $n$, the left side of the display gives the empirical estimator \eqref{eqn-def-eqv-ee}.
Since these random times tend to infinity almost surely as $n\ra\infty$ , we conclude that the empirical estimator
converges almost surely to the right side of the preceding display.
To complete the proof we must identify this right side as $r_k = f(k)/\sum_j f(j) p_j$.

From $\bbP\bigl(\xi(t)>k - 1 \bigr)=\bbP(T_1+\cdots+T_{k}<t)=\int_0^t h_k(u)\,du$, for $h_k$ the density
of $T_1+\cdots+T_{k}$, we have by Fubini's theorem (or partial integration),
\begin{equation}
    \begin{aligned}
 \lambda\int_0^\infty {\mathrm e}^{-\lambda t}\bbP\bigl(\xi(t)>k - 1\bigr)\,dt 
&=\int_0^\infty \int_u^\infty \lambda {\mathrm e}^{-\lambda t}\,dt\,h_k(u)\, du\\
&=\bbE [{\mathrm e}^{-\lambda (T_1+\cdots+T_{k})}]= \prod_{j=1}^{k} \frac{f(j)}{\lambda+f(j)}.
       \end{aligned}
    \label{eqn-xi-bigger-k}
\end{equation}
Furthermore, writing $\bbP\bigl(\xi(t)=k- 1\bigr)$ as the difference of the preceding with $k-2$ and $k-1$,
we obtain
\begin{equation} 
\lambda\int_0^\infty {\mathrm e}^{-\lambda t}\bbP\bigl(\xi(t)=k-1 \bigr)\,dt 
=\frac{\lambda}{\lambda+f(k)}\prod_{j=1}^{k-1} \frac{f(j)}{\lambda+f(j)}.
    \label{eqn-xi-k}
\end{equation}
At $\lambda=\lambda^*$ the right hand side of \eqref{eqn-xi-k} is the same as the limiting proportion $p_k$ 
in \eqref{eqn-pk-limit}, while the right hand side of \eqref{eqn-xi-bigger-k} can be seen to be $p_{>k} := \sum_{j >k} p_j$,
with the help of Fubini's theorem.  Therefore, their quotient is $r_k$ by the succeeding Lemma \ref{lemma-limit-degree-dist-equality}. 

This concludes the proof for the adapted \pa model in which the root has degree 1 higher than in the
original \pa model. The original model starts with two connected nodes $v_1$ and $v_2$, which initially
both have degree 1. We start two independent branching processes of the type described
in Section~\ref{sec:rooted-tree} and attach these to the nodes $v_1$ and $v_2$ as their roots, thus forming a single tree. This union of
the two processes evolves as the original \pa model. If $Z_{t,i}^\phi$ and $Z_{t,i}^\psi$, for $i\in \{1,2\}$, are
the processes counting the number of nodes of degree strictly larger than $k$ and equal to $k$ at time $t$ in the two branching
processes, then $Z_{t,1}^\phi+Z_{t,2}^\phi$ and $Z_{t,1}^\psi+Z_{t,2}^\psi$  are the number of such nodes
in the union of the two branching processes, and 
\begin{align*}
\frac{|\lrbkt{ x\colon  \deg(x,t)>k}|}{ |\lrbkt{ x\colon  \deg(x,t)=k}|} 
&=\frac{\mathrm{e}^{-\lambda^*t}(Z_{t,1}^\phi+Z_{t,2}^\phi)}{\mathrm{e}^{-\lambda^*t}(Z_{t,1}^\psi+Z_{t,2}^\psi)}\\
&\ra \frac{(Y_{1,\infty}+Y_{2,\infty}) \int_{0}^\infty {\mathrm e}^{-\lambda^* t} \bbP(\xi(t) + 1 >k)\, dt}
{ (Y_{1,\infty}+Y_{2,\infty}) \int_{0}^\infty {\mathrm e}^{-\lambda^* t} \bbP(\xi(t) + 1 =k)\, dt},
\end{align*}
almost surely, as $t\ra\infty$, by the first assertion of Proposition~\ref{prop-fraction-as}. 
The almost sure positivity of the variables $Y_{i,\infty}$ allows to cancel them from 
the right side, whence the limit is the same as before.}
\end{proof}

\begin{lemma}
    Suppose  $(p_k)_{k=1}^\infty$ is the limiting degree distribution specified in \eqref{eqn-pk-limit} for the \pa function $ f $.  Then, for all $k\geq 1$,
    \begin{equation}
        \label{eqn-lemma-limit-degree-dist-equality}
        \frac{f(k)}{\sum_{j=1}^\infty p_j f(j)} = \frac{  p_{>k} }{ p_k}.
    \end{equation}
    \label{lemma-limit-degree-dist-equality}
\end{lemma}

\begin{proof}
    Define the auxiliary quantity $q_k$ as the limiting preference towards degree $k$ to be 
    \[q_k = \frac{f(k) p_k }{ \sum_{i=1}^\infty f(i) p_i} = \frac{ f(k) p_k }{\lambda^*}. \]
    Note $q_k = r_k p_k $ and proving \eqref{eqn-lemma-limit-degree-dist-equality} is the same as proving $q_k = p_{>k}$.  

    {We use induction in $k$.}
    For $k=1$, we note $p_1 = \lambda^*/(f(1) + \lambda^*)$, $p_{>1} = 1 - p_1 = f(1)/\bigl(\lambda^* + f(1)\bigr) = f(1) p_1/\lambda^* = q_1$.  Assuming $p_{>k} = q_k$ holds, consider the case of $k+1$.  By \eqref{eqn-pk-limit}, we have $p_{k+1} = f(k) p_k/\bigl(\lambda^* + f({k+1})\bigr)$.
    \begin{align*}
        p_{>k+1} & = p_{>k} - p_{k+1} = q_k - p_{k+1}\\
        & = f(k) p_k \left( \frac{1}{ \lambda^*} - \frac{ 1}{ \lambda^* + f({k+1})} \right)\\
        & = \frac{f({k+1})}{\lambda^*} \frac{f(k) p_k}{ \lambda^* + f({k+1})}\\
        & = \frac{f({k+1}) p_{k+1}}{\lambda^*} = q_{k+1}.
    \end{align*}
    Then by mathematical induction, $p_{>k} = q_k$ holds for all $k \in \bbN_+$.
\end{proof}

\section{Simulation Studies}
In this section we present a numerical illustration of the behavior of the empirical estimator.
We run the simulation for the following \pa functions (after normalization such that $f(1) = 1$ for easy comparisons) 
\begin{gather*}
    f^{\sss (1)}(k) = (k+1/2)/(3/2), \\
    f^{\sss (2)}(k) = k^{2/3}, \\
    f^{\sss (3)}(k) = \sqrt[4]{k+2}/\sqrt[4]{3}.
\end{gather*}
We simulate 1,000 \pa networks of 10,000, 100,000 and 1,000,000 nodes for each of the three functions, so 9,000 networks in total.   
In each experiment, since the model is only identifiable up to scale, we apply the empirical estimation on some finite degrees
and then normalize the obtained estimation such that the preference on degree 1 is $1$ for the sake of easy comparisons. 
In other words, we study $\hat{r}_k(n)/\hat{r}_1(n)$---another rescaled estimate of $f(k)$---instead of $\hat{r}_k(n)$.
We summarize our simulation study in Figure~\ref{fig:ee-all}.  
\begin{figure}[htp]
    \centering
    \includegraphics[width=\textwidth]{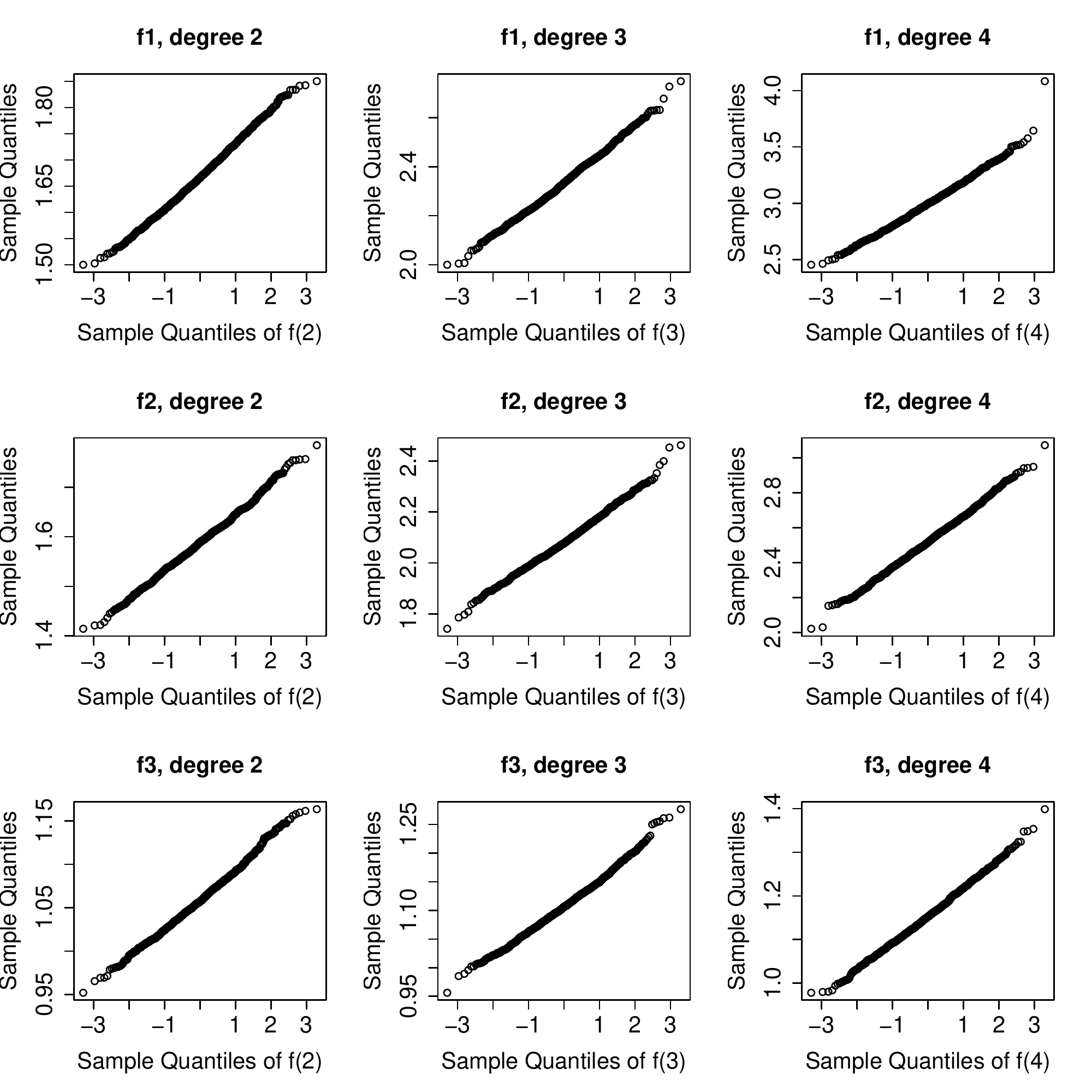}
    \caption{Boxplots of \ee's in different settings.
      The three rows correspond to the \pa functions $f^{\sss (1)}$, $f^{\sss (2)}$, $f^{\sss (3)}$, respectively and the three columns correspond to three network sizes, respectively.  
In each panel the horizontal axis is the degree and the vertical axis correspond to the value of the preference function and a boxplot of the 1,000 estimates.   
The ground truth is marked in red.}
    \label{fig:ee-all}
\end{figure}

These plots suggest the following observations:
\begin{itemize}
    \item[$\rhd$] The estimator is consistent, as our theorem shows.  The quality improves when we have more nodes, hence more observations.   
    \item[$\rhd$] For a fixed number of nodes, the quality of the estimator deteriorates fast when $k$ increases, exemplified by the substantial variability compared to the ground truth. 
    \item[$\rhd$] Even if the \ee has a large variance for large $k$'s, the sample median of $\hat{r}_k$ in each degree $k$ is still remarkably close to the truth. 
    \item[$\rhd$] For a fixed number of nodes it appears that, when the \ee makes larger errors, it is overestimating. 
    \item[$\rhd$] The \ee is not automatically monotone.  However, we can slightly modify the estimator so that it is still consistent but always gives monotone results (cf. \cite{chernozhukov2009improving}).
\end{itemize}
To summarize, the estimator performs as proven in our main result in Theorem~\ref{thm-consist-main}, but the exact performance depends on the true \pa function and the degrees of interest---if the true \pa function increases slowly with respect to the degree, then it is easier to estimate the preference of low degrees, and harder to estimate the preference of high degrees and vice versa.

\subsection{Sample variance study}
We again run 1000 simulations of trees with \pa functions of $f^{\sss (1)}$, $f^{\sss (2)}$, $f^{\sss (3)}$, but now only simulate networks of size 1,000,000.  
We apply the \ee to each simulated network and calculate the sample variance of the 1000 estimates for each given degree up to $70$, 
if the estimate is defined.
The sample variances are plotted against the degrees in Figure~\ref{fig:sample-var}.    Denote the sample variance of the \ee for degree $k$ by $s_k$. Inspection of these plots reveal the following:  
\begin{itemize}
    \item[$\rhd$] It appears that $\log s_k$ grows polynomially with respect to $\log k$.  For the affine \pa function $f^{\sss (1)}$, it looks like $\log s_k$ is an affine function of $\log k$.  
    \item[$\rhd$] The sample variance $s_k$ characterizes, to a certain extent, the difficulty of estimating $r_k$: 
        \begin{itemize}
            \item For small $k$'s, we see that $s^{\sss (3)}_k < s^{\sss (2)}_k < s^{\sss (1)}_k$.   
            \item Then at about $k=17$, the blue line of $s^{\sss (3)}_k$ first crosses the green line of $s^{\sss (2)}_k$, i.e., $s^{\sss (2)}_k < s^{\sss (3)}_k < s^{\sss (1)}_k$.  
            \item The blue line of $s^{\sss (3)}_k$ crosses the red line of $s^{\sss (1)}_k$ at around $k= 18$.  This means $s^{\sss (2)}_k < s^{\sss (1)}_k < s^{\sss(3)}_k$. 
            \item The green line of $s^{\sss (2)}_k$ crosses the red line of $s^{\sss (1)}_k$ at approximately $k = 35$, so from that point on $s^{\sss (1)}_k < s^{\sss (2)}_k < s^{\sss (3)}_k$. 
        \end{itemize} 
    \item[$\rhd$] On the one hand, for small $k$'s, the slower $f$ grows with $k$, the easier it is to estimate $r_k$. This is reflected in the observation that slower $f$ yield lower sample variance for small values of $k$.  On the other hand, for large $k$'s, the faster $f$ grows with $k$, the easier it is to estimate $r_k$. 
   \item[$\rhd$] The shapes of curves of $k\mapsto s_k$ for different $f$'s seems to indicate that the faster $f$ grows with $k$, the slower $\log s_k$ grows with $\log k$.   The seemingly affine relations might be a consequence of the limiting power-law distribution, but it is unclear to us how these relate precisely.    
\end{itemize}
The above observations seem intuitive because for $f$ that grows fast, there are more nodes of high degrees, and this can be expected to yield better results in estimating the preferences of higher degrees.  However, as the total number of nodes is fixed, more nodes of high degrees mean fewer nodes of low degrees.  This results in larger variances in estimating $r_k$ for small $k$'s.  
\begin{figure}[htp]
    \centering
    \includegraphics[width=0.7\textwidth]{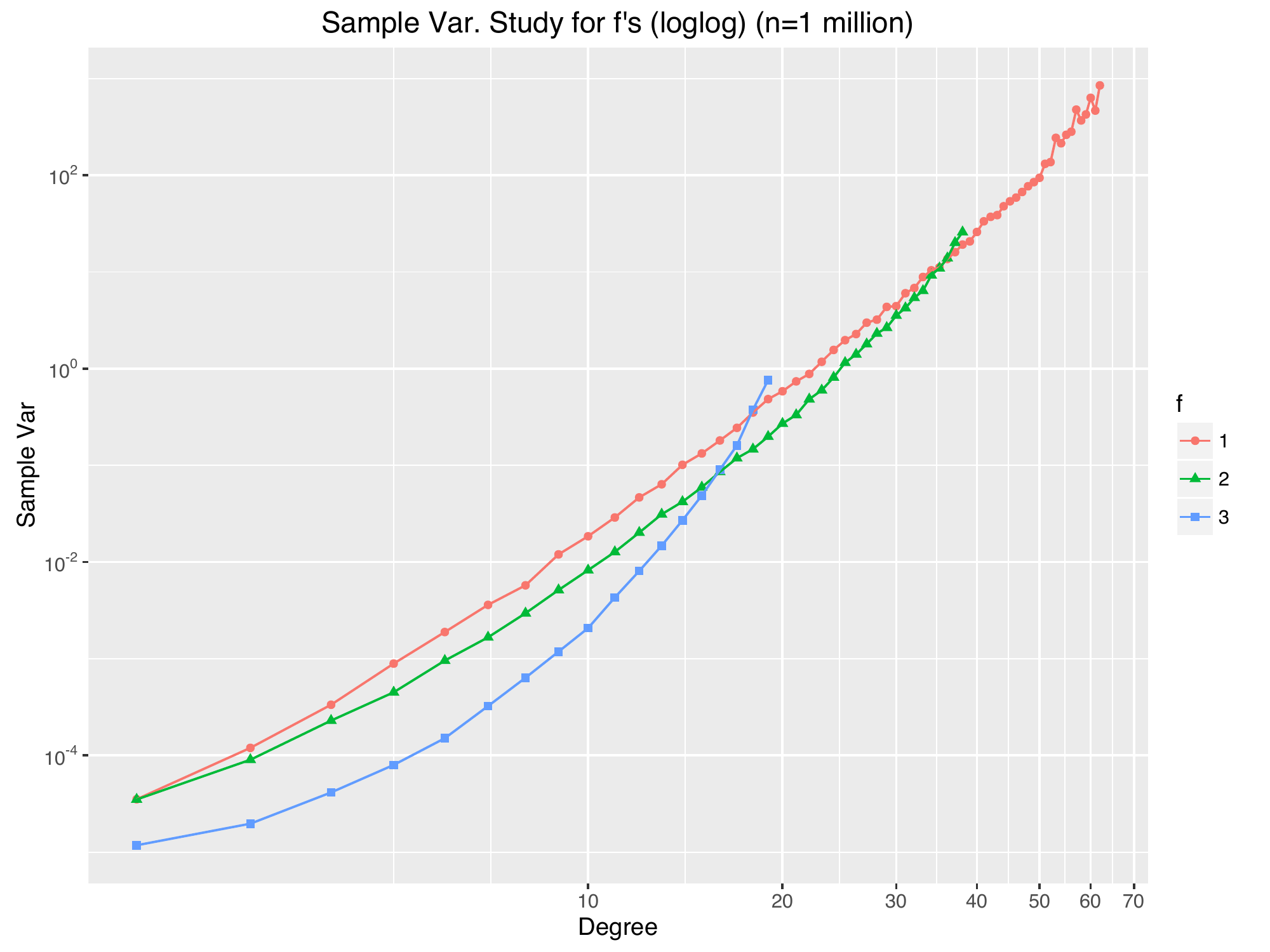}
    \caption{Sample Variance Study of EE.
    Different colors stand for different \pa functions: the red corresponds to $f^{\sss (1)}$, the green $f^{\sss (2)}$ and the blue $f^{\sss (3)}$.  The scale on both axes are in $\log_{10}$.  
    The three lines do not reach the same degrees as there are fewer nodes of high degrees when $f(k)$ grows slower in $k$.}
    \label{fig:sample-var}
\end{figure}

\subsection[Asymptotic normality?]{Asymptotic normality of the estimator with the parametric rate?}
We may wonder what is the asymptotic distribution of $\hat{r}_k(n)$ is, for any fixed $k$ and after proper rescaling, when $n \rightarrow \infty$.   To answer this question, we discuss some simulation results here.  

We fix the number of nodes to be 1,000,000 in all simulated networks for each \pa function.  Then we look at the \ee's in each simulation.  For each $f$, we plot the \qq-plot of each estimator for $k=2,3,4$ against the normal distribution.  The results are summarized in Figure~\ref{fig:qqplot}.   Since the number of nodes is one million, we expect that the limiting distribution should have kicked in, assuming there is indeed a limiting distribution.  The \qq-plots indeed indicate that the \ee's have asymptotic normal distributions.  
\begin{figure}[htp]
    \centering
    \includegraphics[width=0.95\textwidth]{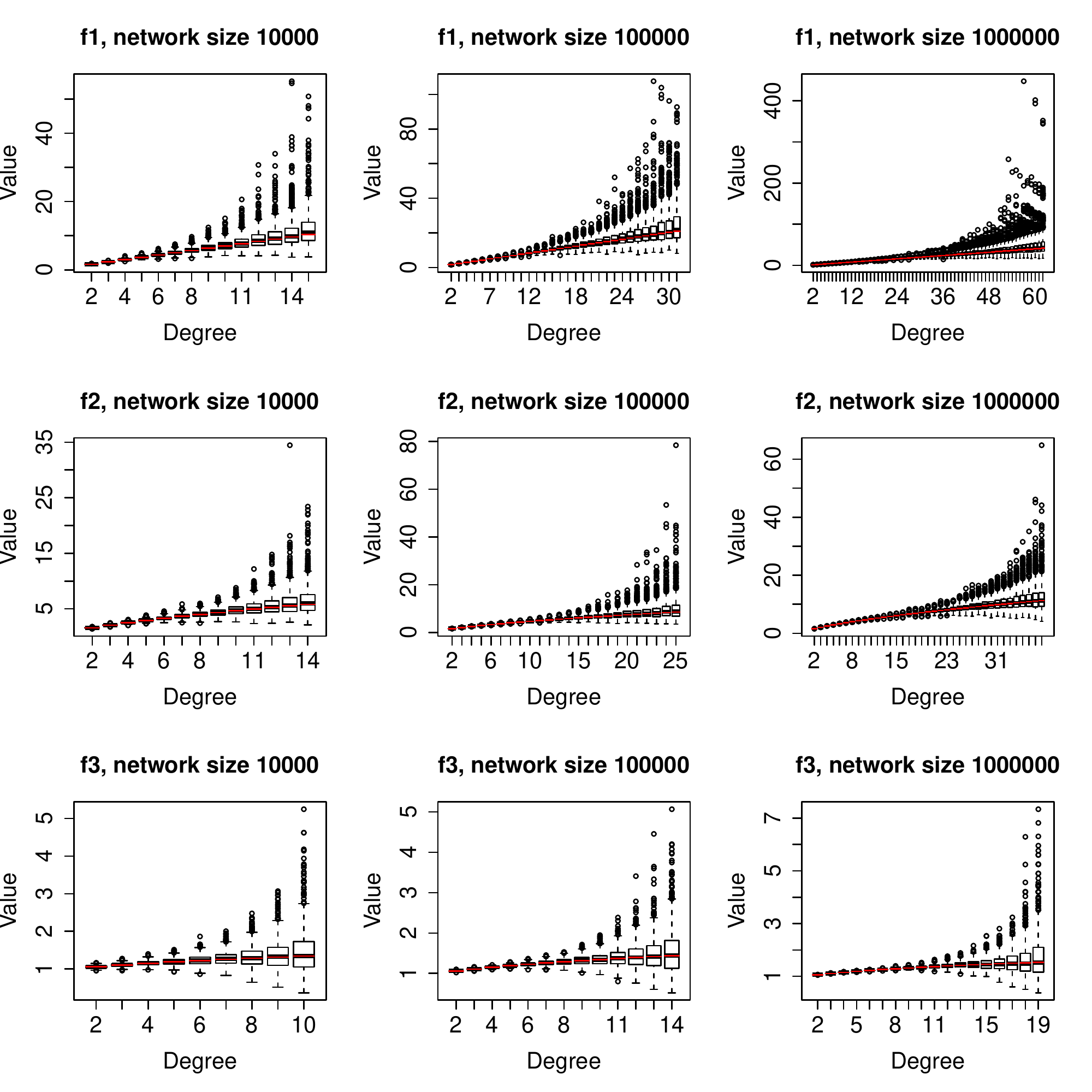}
    \caption[QQ-Plots of Empirical Estimators]{QQ-Plots of $\hat{r}_2(n)$, $\hat{r}_3(n)$ and $\hat{r}_4(n)$ with $n = 10^6$ for $f^{\sss (1)}$, $f^{\sss (2)}$ and $f^{\sss (3)}$
    The rows correspond to the \pa functions $f^{\sss (1)}$, $f^{\sss (2)}$ and $f^{\sss (3)}$, respectively. 
The columns correspond to the degree $k=2,3,4$, on which we conduct our \ee study, respectively.}
    \label{fig:qqplot}
\end{figure}

We conjecture that, for 
fixed $k$,
\begin{equation}
    \label{eqn-rk-sus}
    \sqrt{n} (\hat{r}_k(n)-r_k)\xrightarrow{d} N(0, \sigma_k^2),
\end{equation}
where $\sigma_k^2$ only depend on $f$ and $k$ and $\xrightarrow{d}$ denotes convergence in distribution.  To validate this conjecture, we perform the following simulation study.  We fix the \pa function to be $f^{\sss (2)}$ and run 1,000 simulations for each of the three different network sizes 10,000, 100,000 and 1,000,000 and study the estimators of the preference on degree 2 only.
If \eqref{eqn-rk-sus} is true, then the distribution of $\hat{r}_2(n)$ should stabilize after rescaling them with $\sqrt{n}$.  We summarize the results in Figure~\ref{fig:r2-compar}.

\begin{figure}[htp]
    \centering
    \includegraphics[width=.9\textwidth]{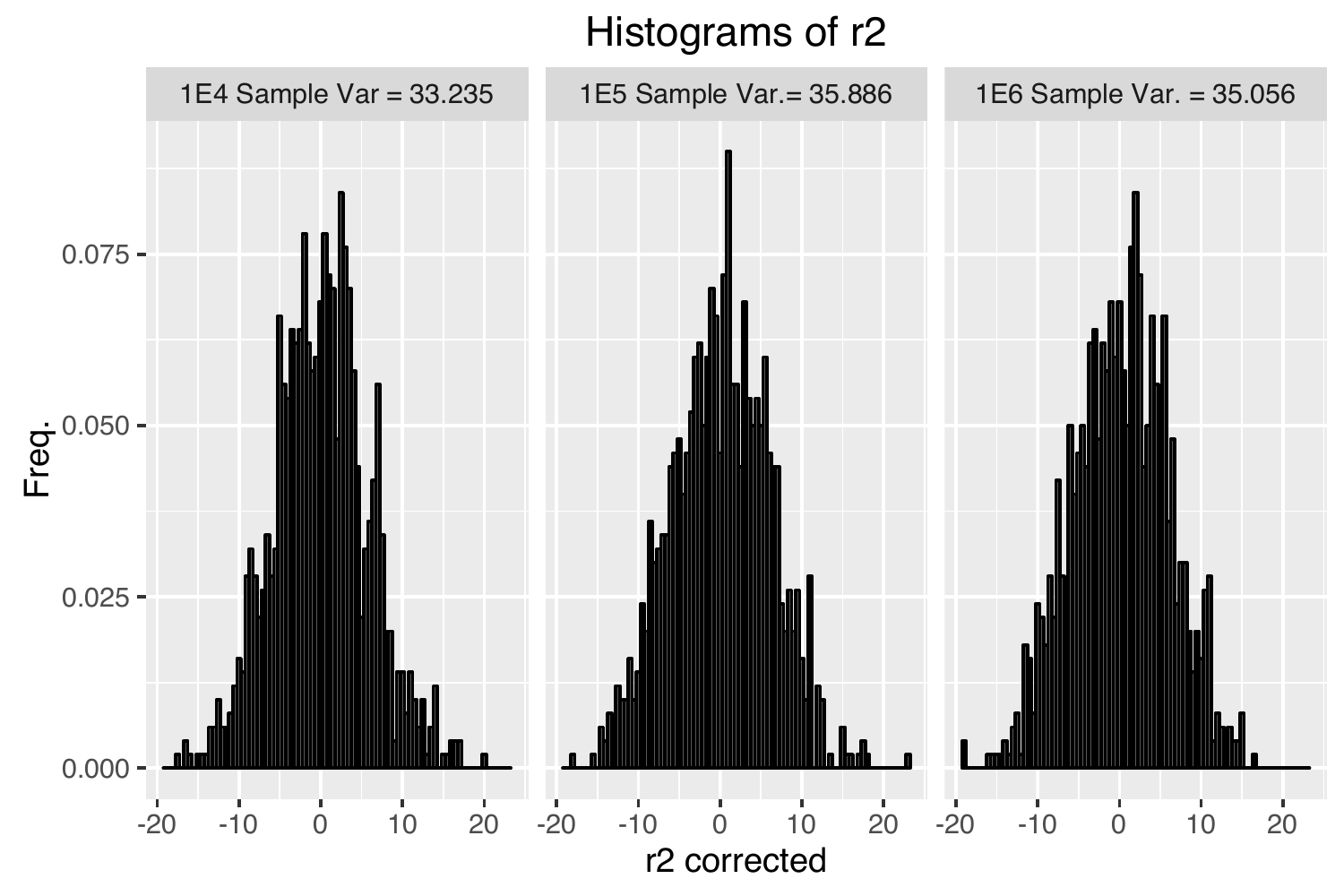}
    \caption[Histogram of Rescaled Empirical Estimator]{Histograms of $\sqrt{n}(\hat{r}_2(n)-r_2)$ for $f^{\sss (2)}$ with different network sizes
    The label \textsf{r2 corrected} on the $x$-axis indicates that the histograms show $\sqrt{n} (\hat{r}_2(n) - r_2)$ (centered and rescaled with the parametric rate) instead of $\hat{r}_2(n)$. The sample variances of each simulation is on top of each subplot. The number of bins is 50.}
    \label{fig:r2-compar}
\end{figure}
The density estimations on the data of $\sqrt{n}(\hat{r}_2(n)-r_2)$ for the three different values of $n$ are displayed in Figure~\ref{fig:r2-density}.  The fact that both the sample variances, histograms and density plots look rather stable after the $\sqrt{n}$-rescaling provides further evidence towards the conjecture in \eqref{eqn-rk-sus}.  

\begin{figure}[htp]
    \centering
    \includegraphics[width=0.7\textwidth]{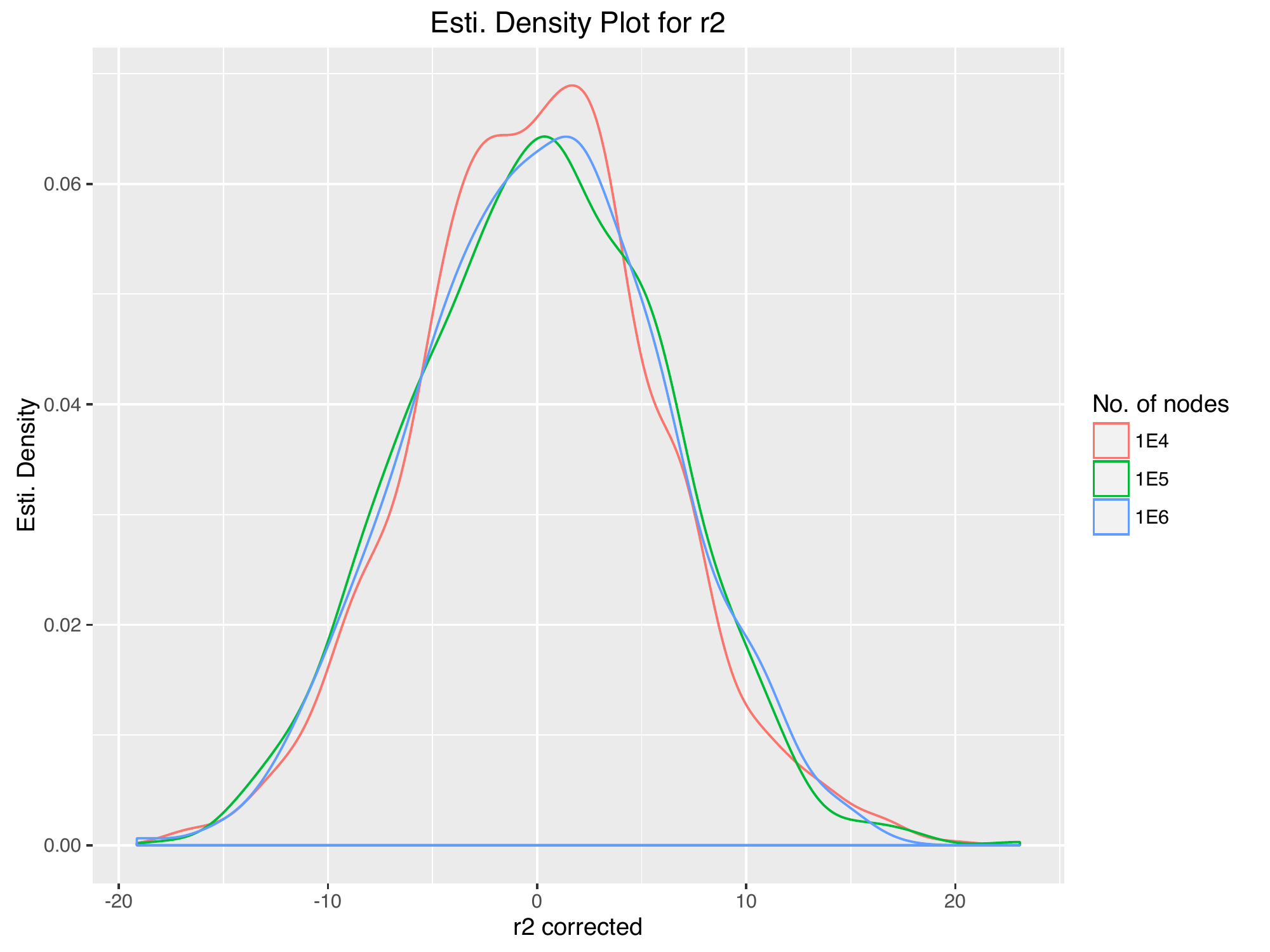}
    \caption{Estimated Density of $\sqrt{n}(\hat{r}_2(n)-r_2)$ with different network sizes $n$}
    \label{fig:r2-density}
\end{figure}

\subsection{Discussion and open problems}
\label{sec-disc-op}
In this paper, we have proposed an empirical estimator for the preferential attachment function in the setting of preferential attachment trees with (sub)linear preferential attachment functions. We rely on an embedding result that views the preferential attachment model as a continuous-time branching process observed at the moments where the branching process has a fixed size. We now discuss some open problems.

\paragraph{Beyond the tree setting.} It would be of interest to extend our analysis to settings where nodes enter the network with more than one edge. This corresponds to general preferential attachment graphs. In this case, the embedding results no longer hold, making the analysis substantially harder. 

\paragraph{The limits of consistency.} 
Our main result in Theorem~\ref{thm-consist-main} shows that the \ee is consistent for every ffixed $k$ fixed. 
For which $k=k_n\ra\infty$ does consistency hold, and how does this depend on the \pa function $f$?
For which norms on $f$ does consistency hold?

\paragraph{Asymptotic normality.} Figures \ref{fig:qqplot}, \ref{fig:r2-compar}  and \ref{fig:r2-density} 
suggest asymptotic normality of  the \ee for fixed values of $k$ at the rate $\sqrt n$. 
How can such a statement be proved? 
It will also be interesting to study the convergence for increasing values of $k$. 
We expect the variance $\sigma_k^2$ in \eqref{eqn-rk-sus} to increase with $k$. This raises a problem 
of bias-variance trade-off when estimating the \pa function for large degrees, much as in ordinary nonparametric
estimation.
\bibliographystyle{abbrvnat}
\bibliography{pam}

\end{document}